\documentclass[12pt]{amsart}
\usepackage{fullpage,url,amssymb,enumerate,colonequals}

\usepackage{mathrsfs} 

\usepackage{MnSymbol}
\usepackage{extarrows}
\usepackage{lscape}
\usepackage[all,cmtip]{xy}
\usepackage[OT2,T1]{fontenc}

\usepackage{color}

\usepackage[
        colorlinks, citecolor=darkgreen,
        backref,
        pdfauthor={Nuno Freitas}, 
]{hyperref}

\usepackage{array, booktabs, comment, multirow, todonotes}

\newcommand{\C}{\mathbb{C}}
\newcommand{\F}{\mathbb{F}}
\newcommand{\Fbar}{{\overline{\F}}}
\newcommand{\Q}{\mathbb{Q}}
\newcommand{\Z}{\mathbb{Z}}
\newcommand{\Qbar}{{\overline{\Q}}}
\newcommand{\rhobar}{{\overline{\rho}}}
\newcommand{\calO}{\mathcal{O}}

\newcommand{\ff}{\mathfrak{f}}
\newcommand{\fp}{\mathfrak{p}}
\newcommand{\fP}{\mathfrak{P}}
\newcommand{\fq}{\mathfrak{q}}
\newcommand{\fN}{\mathfrak{N}}

\DeclareMathOperator{\Aut}{Aut}

\DeclareMathOperator{\End}{End}
\DeclareMathOperator{\Frob}{Frob}
\DeclareMathOperator{\Gal}{Gal}
\DeclareMathOperator{\JL}{JL}
\DeclareMathOperator{\Norm}{Norm}
\DeclareMathOperator{\tr}{tr}
\DeclareMathOperator{\T}{\mathbb{T}}
\DeclareMathOperator{\Tr}{Tr}

\newcommand{\vv}{\upsilon}
\newcommand{\GL}{\operatorname{GL}}
\newcommand{\PGL}{\operatorname{PGL}}
\newcommand{\PSL}{\operatorname{PSL}}
\newcommand{\SL}{\operatorname{SL}}

\numberwithin{equation}{section}

\newtheorem{theorem}{Theorem}[section]
\newtheorem{lemma}[theorem]{Lemma}
\newtheorem{corollary}[theorem]{Corollary}
\newtheorem{proposition}[theorem]{Proposition}

\theoremstyle{definition}

\theoremstyle{remark}
\newtheorem{remark}[theorem]{Remark}

\definecolor{darkgreen}{rgb}{0,0.5,0}

\begin{document}

\title{Some extensions of the modular method and Fermat equations
of signature $(13,13,n)$}

\author{Nicolas Billerey}
\address{Universit\'e Clermont Auvergne, CNRS, LMBP, F-63000 Clermont-Ferrand, France.
}
\email{nicolas.billerey@uca.fr}

\author{Imin Chen}

\address{Department of Mathematics, Simon Fraser University\\
Burnaby, BC V5A 1S6, Canada } \email{ichen@sfu.ca}

\author{Lassina Demb\'el\'e}

\address{Department of Mathematics, University of Luxembourg, Maison du Nombre, Avenue de la Fonte, Esch-sur-Alzette L-4364, Luxembourg} 
\email{lassina.dembele@gmail.com}

\author{Luis Dieulefait}

\address{Departament d'Algebra i Geometria,
Universitat de Barcelona,
G.V. de les Corts Catalanes 585,
08007 Barcelona, Spain}
\email{ldieulefait@ub.edu}

\author{Nuno Freitas}
\address{
Instituto de Ciencias Matem\'aticas, CSIC, 
Calle Nicol\'as Cabrera
13--15, 28049 Madrid, Spain}
\email{nuno.freitas@icmat.es}

\date{\today}


\thanks{We acknowledge the financial support of ANR-14-CE-25-0015 Gardio (N.~B.), an NSERC Discovery Grant (I.~C.), and the grant {\it Proyecto RSME-FBBVA $2015$ Jos\'e Luis Rubio de Francia} (N.~F.)}

\maketitle

\begin{abstract} 
We provide several extensions of the modular method which were motivated by the problem of completing previous work to prove that, for any integer $n \geq 2$, the equation
\[ x^{13} + y^{13} = 3 z^n \]
has no non-trivial primitive solutions. In particular, we present four elimination techniques which are based on: (1) establishing reducibility of certain residual Galois representations over a totally real field; (2) generalizing image of inertia arguments to the setting of abelian surfaces; (3) establishing congruences of Hilbert modular forms without the use of often impractical Sturm bounds; and (4) a unit sieve argument which combines information from classical descent and the modular method.

The extensions are of broader applicability and provide further evidence that it is possible to obtain a complete resolution of a family of generalized Fermat equations by remaining within the framework of the modular method. As a further illustration of this, we complete a theorem of Anni-Siksek to show that, for $\ell, m\ge 5$, the only primitive solutions to the equation $x^{2\ell} + y^{2m} = z^{13}$ are trivial.
\end{abstract}

\section{\bf Introduction}
In \cite{BCDF}, it was shown that it is possible to get an optimal result for a family of generalized Fermat equations, specifically $x^5 + y^5 = 3 z^n$ for $n \ge 2$ an integer, using a refined modular method with the multi-Frey approach over totally real fields. These methods were also applied to the equation
\begin{equation}\label{fte:13-13-7}
x^{13} + y^{13} = 3 z^p,
\end{equation}
where $p$ is prime, but failed for $p = 7$; this failure is due to the 
presence of five specific Hilbert modular forms after level lowering which the authors 
were not able to eliminate. Four of the obstructing forms are newforms of weight $2$ defined over the real cubic subfield of $\Q(\zeta_{13})$, the cyclotomic field of $13$th root of unity,
whose level is the unique ideal of norm $2808=2^3\cdot 3^3 \cdot 13$. The remaining one, which is the most relevant to us, is denoted by $g$; it is a newform defined  over $\Q(\sqrt{13})$ of parallel weight~$2$, level $\fN = (2^3 13)$ and field of coefficients $\Q(\sqrt{2})$ appearing in \cite[Proposition 9]{BCDF}.

The techniques in this paper provide extensions of the modular method which were initially motivated by the problem of eliminating these five obstructing forms, but are of broader applicability. In particular, we successfully treat the remaining case $p=7$ of equation~\eqref{fte:13-13-7}. 

Our first step is to reduce the resolution of~\eqref{fte:13-13-7} for $p=7$ to the problem of 
dealing with the form~$g$. For this, we consider the Frey curve $F = F_{a,b}$ defined in \cite[\S 7.2]{BCDF}. The refined elimination technique from~\cite[\S 7.3]{BCDF} applied to $F$ 
shows that the obstruction to the modular method arises from the four obstructing newforms defined over the real cubic subfield of $\Q(\zeta_{13})$, with level of norm $2808$. 
We are then able to show that each of these forms has a reducible mod~$\fp$ representation for some prime~$\fp \mid 7$ (see Proposition~\ref{P:reducible}). 
Since the mod~$7$ representation attached to $F$ is irreducible by assumption, this allows us to discard these four forms (see Sections~\ref{S:elimination} and~\ref{S:redg}). 
However, establishing the reducibility property for the four obstructing forms requires a bit of care since the usual approach, which involves using the Sturm bound, is 
not computationally feasible in the setting of Hilbert modular forms.

The remaining obstructing form~$g$ arises for the Frey curve $E = E_{a,b}$ defined in \cite[\S 7.1]{BCDF}, more specifically, when trying to prove \cite[Theorem~7]{BCDF} for $p = 7$. We will introduce three different methods for discarding this form.

The first method eliminates~$g$ by extending an `image of inertia argument' to the setting of abelian surfaces. For this, we will combine modularity of an abelian surface with the study of the inertial types of~$g$ at suitable primes. 

The second method is to show that~$g$ is congruent modulo $7$ to the Hilbert newform associated to the Frey curve corresponding to the trivial solution $(1, -1, 0)$; see also \cite[Remark 7.4]{BCDF}. In general, the proof of such congruences requires computing Hecke eigenvalues up to the Sturm bound, which is often not practical for Hilbert modular forms, as is the case in our situation.  Instead, we exploit a multiplicity one phenomenon for the residual Galois representations attached to the Hilbert newforms in our situation, to establish the desired congruence.

The third method is based on combining the study of units in classical descent with
local information coming from the modular method together with restrictions on 
the solutions coming from the multi-Frey approach in~\cite{BCDF}. 

As a consequence, we obtain a complete resolution of the generalized Fermat equations 
in~\eqref{fte:13-13-7}. More precisely, we establish the following result.
\begin{theorem}\label{T:main} 
For all integers $n \geq 2$, there are no integer solutions $(a,b,c)$ to the equation 
\begin{equation*} \label{E:13137}
 x^{13} + y^{13} = 3 z^n
\end{equation*}
such that $abc \neq 0$ and $\gcd(a,b,c) = 1$.
\end{theorem}
An integer solution $(a,b,c)$ to \eqref{E:13137} is said to be a primitive solution if $\gcd(a,b,c) = 1$, and trivial if $abc = 0$.

It is natural to wonder whether one could use Chabauty methods to solve~\eqref{E:13137} for $p=7$ by determining the rational points on the hyperelliptic curves associated with \eqref{E:13137} from the constructions in \cite{DahmenSiksek}. This would necessitate computing the $2$-Selmer groups of Jacobians of genus 3 hyperelliptic curves of the form $y^2 = x^7 + a$, where $a$ belongs to the maximal totally real sextic subfield $F$ of the field $\Q(\zeta_{13})$. Because some  values of $a$ which need to be treated are not seventh powers in $F$,  such calculations would require working over the extensions $F(\sqrt[7]{a})$ which are of absolute degree~$42$. Even under GRH, this would be extremely challenging with current methods. Moreover, it would only lead to a conditional resolution of \eqref{E:13137}.

The computations required to support the proof of our theorems were performed using {\tt Magma}~\cite{magma} (V2.25-5), mainly the Hilbert Modular Forms Package (see~\cite{DV13, GV11}). 
The program and output transcript files are available at \cite{programs} (see {\tt Read\_Me.txt}  for a description of the files). The Dokchister-Doris package for computing the conductor of genus $2$ curves is available at~\cite{Dor17}. 

\subsection*{Acknowledgements} We would like to thank T.~Dokchitser and C.~Doris for making their algorithm available to us,
and for helpful email correspondence. We also thank M.~Stoll for conversations regarding the use of the 
Chabauty method for the problem treated in this paper, A.~Pacetti for discussions concerning Sturm bounds for Hilbert modular forms and N. Mascot for (double-)checking part~(\ref{item:1}) of Theorem~\ref{T:modularity} using the algorithm in~\cite{cmsv}. We finally thank the anonymous referees for their careful reading and useful remarks.

\section{The elimination step away from reducible primes} 
\label{S:elimination}

In this section, we recall the elimination step of the modular method and discuss ways in which it can fail to yield a contradiction. 
The discussion will center on equation~\eqref{E:13137} for convenience, but the same principles apply more generally.

Let $(a,b,c)$ be a primitive solution to~\eqref{E:13137} and~$E_{a,b}/K$ a Frey curve attached to it, where $K$ is a totally real number field.
Also, let $\rhobar_{E_{a,b},p}$ be the mod~$p$ representation attached to~$E_{a,b}$. By assumption $\rhobar_{E_{a, b}, p}$ is an irreducible
representation which is modular. An application of standard results on level lowering implies that there is a Hilbert newform $f$ of 
level $\fN$ and weight $2$, with field of coefficients $\Q_f$, such that
\begin{equation}\label{E:iso}
 \rhobar_{E_{a,b},p} \simeq \rhobar_{f,\fp}
\end{equation}
where $\fp \mid p$ is a prime in $\Q_f$ and $\rhobar_{f, \fp}$ the mod~$\fp$ representation attached to $f$.
In practice, the level $\fN$ is a concrete ideal dividing the conductor of $E_{a,b}$, 
independent of~$(a,b,c)$ and of `small' norm. So the new subspace $S_2(\fN)^{\rm new}$ 
of Hilbert cusp forms of level $\fN$, weight $2$ and trivial character
is accessible via the Hilbert Modular Forms package in \verb|Magma|~\cite{magma}.

Given $f \in S_2(\fN)^{\rm new}$, a candidate newform, 
a crucial step in the modular method consists of showing that the isomorphism \eqref{E:iso} does not occur for {\it any} prime $\fp$
of $\Q_f$ above $p$. We say that we {\it have eliminated the form~$f$ for the exponent~$p$} if we can successfully complete that
step. By eliminating all the newforms in $S_2(\fN)^{\rm new}$, we obtain a contradiction to the fact that $\rhobar_{E_{a,b},p}$ is
modular, thus solving~\eqref{E:13137} for the exponent~$p$.

By taking traces at Frobenius elements on both sides of~\eqref{E:iso} one gets
\begin{equation} \label{E:frobCong}
a_\fq(E_{a,b}) \equiv a_\fq(f) \pmod{\fp}\,\, \text{for all primes}\,\,\fq \nmid p\cdot \mathrm{cond}(E_{a,b}),
\end{equation}
where~$a_{\fq}(f)$ is the Hecke eigenvalue of $f$ at $\fq$, and \(a_{\fq}(E_{a,b})\) 
the trace of Frobenius of $E$ at~$\fq$.

Let~$q \neq p$ be a rational prime coprime to the level~$\fN$ of~$f$.
Suppose that~$E_{a,b}$ has good reduction at each prime ideal~$\fq$ dividing~$q$ in~$K$,
and define the quantity
\[
B_q(E_{a,b},f)=\gcd\left( \left\{ \Norm\left( a_{\fq}(E_{a,b})-a_{\fq}(f)\right): \fq\mid q \right\} \right).
\]
Note that~\eqref{E:frobCong} implies the exponent~$p$ divides $q B_q(E_{a,b},f)$. 

Suppose instead that $E_{a,b}$ has multiplicative reduction at a prime $\fq \mid q$. Then, the isomorphism~\eqref{E:iso}
implies that the form~$f$ satisfies level raising conditions, hence we have the congruence:
\begin{equation} \label{E:multCong}
  a_{\fq}(f)\equiv \pm(\Norm(\fq)+1)\pmod{\fp}.
\end{equation}
In the elimination step, we find rational primes~$q$ such that, for all $a,b$ mod~$q$ not both zero, the Frey 
curve~$E_{a,b}$ has either good reduction at all~$\fq \mid q$ 
or multiplicative reduction at all~$\fq \mid q$. 
Then, for a newform~$f \in S_2(\fN)^{\rm new}$, we
define
\[ A_q(f) : = q\prod_{(0,0) \neq (a,b) \in \F_q^2} B_q(E_{a,b},f) \cdot \prod_{\fq \mid q} \Norm(a_{\fq}(f)^2-(\Norm(\fq)+1)^2),\] 
where the first part accounts for the possibility of $E$ having good reduction at~$q$ whilst the second product for the case of multiplicative reduction.
The quantity $A_q(f)$ is independent of~$p$, so once we find one prime~$q$ such that $A_q(f) \neq 0$, we have eliminated~$f$ for large~$p$.
For most~$f$, trying a few auxiliary primes often suffices to find such a $q$. This is then enough to 
eliminate~$f$ for all exponents~$p \nmid A_q(f)$. Furthermore, by using several primes~$q_1,\ldots, q_s$, 
we eliminate~$f$ for all $p \nmid \gcd(A_{q_1}(f),\ldots, A_{q_s}(f))$. 

For a given newform~$f$, the elimination step leads to a set of primes~$P_f$, which cannot be excluded by this process.
This set usually consists of a few small primes, but it can also be empty or equal to the set of the prime numbers; 
the latter only occurs when all the auxiliary primes we tried gave $A_q(f)=0$. 

In~\cite[\S 7.3]{BCDF}, four of the authors introduced the following refined elimination technique to deal with 
the remaining primes in~$P_f$. The key observation of this refinement is that~\eqref{E:frobCong} and~\eqref{E:multCong}
imply $p \mid A_q(f)$ but the converse is not true. So, for the primes~$p \in P_f$, they test for the congruences~\eqref{E:frobCong} 
and~\eqref{E:multCong}, which result from the isomorphism~\eqref{E:iso}, directly. In other words, for each prime $\fp \mid p$ in $\Q_f$,
they check the following congruences by direct computations:
\begin{enumerate}[(i)]
\item\label{item:i} for all~$\fq \mid q$ of good reduction for~$E_{a,b}$ we have $a_{\fq}(f)\equiv a_{\fq}(E_{a,b})\pmod{\fp}$;
\item\label{item:ii} for all~$\fq \mid q$ of multiplicative reduction for~$E_{a,b}$, we have $a_{\fq}(f)\equiv \pm(\Norm(\fq)+1)\pmod{\fp}$;
\end{enumerate}
This can be done as long as we can factor the primes~$p \in P_f$ in the field~$\Q_f$. If (\ref{item:i}) and (\ref{item:ii}) fail for~$\fp \mid p$
then we have eliminated the pair $(f,\fp)$. Excluding all such pairs then eliminates~$f$ for the exponent~$p$. 
In~\cite[\S 7.3]{BCDF}, this refinement allowed for the resolution of \eqref{E:13137} for~$p=5,11$. 

Let us now discuss one way in which the argument above can fail. Suppose for some form~$f$ and a prime~$\fp_0 \mid p \in P_f$ in $\Q_f$, the mod~$\fp_0$ representation is reducible and, moreover, it satisfies $\rhobar_{f,\fp_0}^{ss} \simeq \chi_p \oplus 1$, where~$\chi_p$ 
is the mod~$p$ cyclotomic character. Then, this form satisfies $a_{\fq}(f)\equiv \Norm(\fq)+1 \pmod{\fp_0}$, and
therefore congruence (\ref{item:ii}) holds with $\fp = \fp_0$. Thus we are unable to eliminate the pair $(f,\fp_0)$ using the refined elimination technique.

However, we note that the isomorphism~\eqref{E:iso} was obtained via level lowering theorems, 
which require the representations involved to be residually irreducible. So, in testing for the congruences (\ref{item:i}) and (\ref{item:ii}), 
we only need to consider those pairs $(f, \fp)$, where $\fp \mid p$ is a prime in $\Q_f$ such that $\rhobar_{f, \fp}$ is irreducible.
This allows us to circumvent the reducibility issue raised in the previous paragraph. We will see in Section~\ref{S:reducible} that 
one of the reasons why equation~\eqref{E:13137} does not follow directly from the methods in~\cite{BCDF},  for $p=7$, 
is the presence of  reducible pairs~$(f,\fp_0)$. 

We note that this idea of avoiding reducible primes~$\fp_0$ 
appeared in the work of Bugeaud-Mignotte-Siksek~\cite{BMSI}, where they work with classical forms over~$\Q$. In that setting, a practical Sturm bound is available, allowing them to prove isomorphism 
of the form $\rhobar_{f,\fp_0}^{ss} \simeq \chi_p \oplus 1$.
For our purpose, due to the lack of practical Sturm bounds for Hilbert
modular forms, we will establish the existence of our reducible pairs~$(f,\fp_0)$ using base change arguments.

\section{A reducibility result for Hilbert modular forms}
\label{S:reducible}

In this section, we let $K = \Q(b)$ be the real cubic subfield of $\Q(\zeta_{13})$, where $b^3 + b^2 - 4b + 1 = 0$. 
We let $\fq_{13}$ be the unique ramified prime in~$K$, which lies above~$13$. We note that $2$ and $3$ are inert in $K$. 

Let $\fN = 2 \cdot 3 \cdot \fq_{13}$, and write  $S_2(\fN)^{\rm new}$ for
the new subspace of Hilbert cusp forms of level $\fN$, weight $2$ and trivial character. 
For $f \in S_2(\fN)^{\rm new}$ we let $\Q_f$ denote the field of coefficients of~$f$. 
For a newform~$f$ and $\fp \mid p$ a prime in $\Q_f$, write $\rhobar_{f,\fp}:\,G_K \to \GL_2(\F_\fp)$ the mod $\fp$ representation attached to $f$. Recall that $\rhobar_{f,\fp}$ can ramify only at primes dividing~$\fN$ and $p$.

For certain newforms $f \in S_2(\fN)^{\rm new}$,
we will establish the reducibility of $\rhobar_{f,\fp}$, where $\fp \mid 7$ is a ramified prime in~$\Q_f$. 
We remark that, for cubic fields (or higher degree fields), there are no effective Sturm bounds available in the
literature. Furthermore, even in the quadratic case, such effective bounds are too large to be of any practical use (see Section~\ref{sec:global-proof}).

Instead, we present a method for establishing reducibility in certain situations when the residual representation is a base change from $\Q$. 
The method is presented in the particular case of interest, but could apply in other similar settings.

The space $S_2(\fN)^{\rm new}$ has dimension $181$, and decomposes into 15 Hecke constituents as:
\begin{equation}
 \label{E:degrees}
181 = 1 + \underbrace{1 + 1 + 1} +\, 3 + 6 + 6 + 12 + 15 + 18 + 18 + 21 + 24 + 27 + 27. 
\end{equation}
More specifically, there are:
\begin{enumerate}[1.]
\item Four newforms with rational Hecke eigenvalues. One of them corresponds to the base change of the elliptic curve with Cremona label $78a1$. The other three correspond to non base change elliptic curves over~$K$ which are permuted by the action of $\Gal(K/\Q)$.
\item One newform with Hecke eigenvalues in the real cubic field $\Q(\zeta_7)^+ \subset \Q(\zeta_7)$.

\item Ten newforms whose Hecke eigenvalue fields have degrees $6, 12, 15, 18, 21, 24$ or $27$. For each of the fields $\Q_f$, there
is a {\it unique} subfield $E$ such that $\Q_f$ is a ray class field of degree $3$ over $E$. We have summarized the structure of these
Hecke eigenvalue fields in Table~\ref{table:coeff-fields}. When $[\Q_f:\Q] \le 21$, we have given a
generator for the conductor $\mathfrak{f}_{\Q_f}$ of $\Q_f/E$, and its factorization. For $[\Q_f:\Q]>21$,
only the the factorization is provided. 
\end{enumerate}

\begin{table}\tiny
\caption{The Hecke eigenvalue fields for the newforms in $S_2(\fN)^{\rm new}$ over the cubic subfield
$K$ of $\Q(\zeta_{13})^+$ which are a cubic extension of a unique subfield.}
\label{table:coeff-fields}
\begin{tabular}{ >{$}c<{$} >{$}c<{$} }\toprule

\multicolumn{2}{c}{$[\Q_f:\Q] = 6$} \\ \midrule
E & \Q(w) := \Q[x]/(x^2-x-7) \\
\mathfrak{f}_{\Q_f} & (2 w + 2) = \fp_2 \fp_7 \\ \midrule \\

\multicolumn{2}{c}{$[\Q_f:\Q] = 6$} \\ \midrule
E & \Q(w) := \Q[x]/(x^2-x-1) \\
\mathfrak{f}_{\Q_f} & ( -8 w + 6) = \fp_2 \fp_{19} \\ \midrule \\

\multicolumn{2}{c}{$[\Q_f:\Q] = 12$} \\ \midrule
E & \Q(w) := \Q[x]/(x^4 + 2x^3 - 9x^2 - 12x - 3) \\
\mathfrak{f}_{\Q_f} & ( -w^2 - 3 w + 3) = \fp_2 \fp_7 \\ \midrule \\

\multicolumn{2}{c}{$[\Q_f:\Q] = 15$} \\ \midrule
E & \Q(w) := \Q[x]/(x^5 - 47x^3 - 105x^2 + 70x + 1) \\
 \mathfrak{f}_{\Q_f} & (1/8(-3w^4 - w^3 + 138w^2 + 369w + 9)) = \fp_2 \\ \midrule \\

\multicolumn{2}{c}{$[\Q_f:\Q] = 18$} \\ \midrule
E & \Q(w) := \Q[x]/(x^6 - x^5 - 58x^4 - 79x^3 + 856x^2 + 2865x + 2495) \\
 \mathfrak{f}_{\Q_f} & (-8w^5 + 32w^4 + 363w^3 - 425w^2 - 5425w - 7264) = \fp_2 \fp_{31} \\ \midrule \\

\multicolumn{2}{c}{$[\Q_f:\Q] = 18$} \\ \midrule
E & \Q(w) := \Q[x]/(x^6 - 28x^4 - 2x^3 + 104x^2 + 104x + 28) \\
 \mathfrak{f}_{\Q_f} & (1/2(-w^4 + 2w^3 + 26w^2 - 48w - 40)) = \fp_2 \fp_7 \\ \midrule \\

\multicolumn{2}{c}{$[\Q_f:\Q] = 21$} \\ \midrule
E & \Q(w) := \Q[x]/(x^7 + x^6 - 58x^5 + 172x^4 + 84x^3 - 744x^2 + 512x + 144) \\
 \mathfrak{f}_{\Q_f} & (1/8(-17w^6 - 77w^5 + 726w^4 - 296w^3 - 2904w^2 + 2144w + 592)) =\\ & \fp_2 \fp_7\fp_7' \\ \midrule \\

\multicolumn{2}{c}{$[\Q_f:\Q] = 24$} \\ \midrule
E & \Q(w) := \Q[x]/(x^8 + 3x^7 - 53x^6 - 84x^5 + 1018x^4 + 190x^3 - 6992x^2\\ &{} + 5440x + 5060) \\
 \mathfrak{f}_{\Q_f} & \fp_2 \fp_7 \\ \midrule \\

\multicolumn{2}{c}{$[\Q_f:\Q] = 27$} \\ \midrule
E & \Q(w) := \Q[x]/(x^9 + 2x^8 - 96x^7 - 99x^6 + 2894x^5 + 1462x^4 - 32500x^3\\ &{} +  2240x^2 + 119964x - 102212) \\
 \mathfrak{f}_{\Q_f} & \fp_2 \fp_7 \fp_{13} \\ \midrule \\

\multicolumn{2}{c}{$[\Q_f:\Q] = 27$} \\ \midrule
E & \Q(w) := \Q[x]/(x^9 - x^8 - 90x^7 + 106x^6 + 2878x^5 - 4048x^4 - 38316x^3\\ &{} +  61316x^2 + 172200x - 284688) \\
 \mathfrak{f}_{\Q_f} & \fp_2 \fp_{1153} \\
\bottomrule
\end{tabular}
\end{table}

Let $\sigma$ be a generator for $\Gal(K/\Q)$.
Let $f$ be a newform in a Hecke constituent in $S_2(\fN)^{\rm new}$, and $d_f$ the dimension of this constituent. 
By cyclic base change and Shimura~\cite{shi78}, there exists a newform $f^\sigma \in S_2(\fN)^{\rm new}$ such that 
$\Q_{f^\sigma} =  \Q_f$ and $$a_\fq(f^\sigma) = a_{\fq^\sigma}(f),\,\text{for all primes }\,\fq \,\text{ of } K,$$
so that $\Gal(K/\Q)$ acts on the Hecke constituents.
Since $[K : \Q] = 3$ the orbits of this action are either of size 1 or 3. The Hecke constituents of dimension $1$ split into two orbits under this action: there is one orbit of size $1$ which
corresponds to the base change of the elliptic curve with Cremona label $78a1$; the other orbit of size 3, underlined in~\eqref{E:degrees},
corresponds to a set of $3$ elliptic curves permuted by $\Gal(K/\Q)$.

Assume that $d_f \ge 3$. Then, since the number of Hecke constituents of dimension $d_f \geq 3$ is at most $2$, we conclude these constituents must 
be fixed by the action of $\Gal(K/\Q)$. Thus, $f$ and~$f^\sigma$ belong to the same constituent. However, a quick inspection of the first few Hecke 
eigenvalues shows that $f$ is not a base change. Therefore, there exists a non-trivial element 
$\tau \in \Gal( \Q_f/E)$ such that
$$a_{\fq}(f^\sigma) = a_{\fq^\sigma}(f) = \tau(a_{\fq}(f)),\,\text{for all primes}\,\,\fq \text{ of } K.$$

\begin{lemma}\label{lem:ramified-p7}
For $d = 6, 12, 18, 21, 24$ or $27$, let $S_d \subset S_2(\fN)^{\rm new}$ be the unique Hecke constituent of dimension $d$ containing a newform $f = f_d$ with $7 \mid 
\Norm(\mathfrak{f}_{\Q_f})$, where 
$\mathfrak{f}_{\Q_f}$ is as in Table~\ref{table:coeff-fields}. 
Then, for every prime $\fp_0 \mid \gcd(\mathfrak{f}_{\Q_f}, (7))$, we have $\Norm(\fp_0) = 7$, and $\fp_0$ totally ramifies in $\Q_f/E$.
\end{lemma}

\begin{proof} This follows from a computation in {\tt Magma}.
\end{proof}

\begin{proposition}\label{P:reducible}
For $d = 12, 21, 24$ or $27$, let $f \in S_d$ and $\fp_0$ in $\Q_f$ be as 
in Lemma~\ref{lem:ramified-p7}. 
Suppose that $\rhobar_{f, \fp_0}$ ramifies at all the primes dividing the level~$\fN$.
Then, $\rhobar_{f, \fp_0}$ is reducible.
\end{proposition}
\begin{proof} Let $f$ and $\fp_0$ be as in the statement. Suppose
that $\rhobar_{f, \fp_0}$ is irreducible.

Recall from the discussion preceding Lemma~\ref{lem:ramified-p7} that for a non-trivial $\sigma \in \Gal(K/\Q)$ there is $\tau \in \Gal(\Q_f/E)$ such that
$$a_{\fq^\sigma}(f) = \tau(a_{\fq}(f)),\,\text{for all primes}\,\,\fq.$$
Since $\fp_0$ is totally ramified in $\Q_f/E$, and the residue field is $\F_{\fp_0} = \F_7$, $\tau: \Q_f \to \Q_f$ gives rise to the identity map $\F_7 \to \F_7$. It follows that
\begin{equation} \label{E:congs}
a_{\fq^\sigma}(f) \equiv a_{\fq}(f) \pmod{\fp_0},\,\,\text{for all primes}\,\,\fq.
\end{equation} 
Let $\tilde{\sigma} \in \Gal(\Qbar/\Q)$ be a lift for $\sigma$. Consider the conjugate representation $\rhobar_{f,\fp_0}^\sigma$ which satisfies
$$\rhobar_{f,\fp_0}^\sigma (\Frob_\fq) := \rhobar_{f,\fp_0}(\Frob_{\fq^\sigma}) = \rhobar_{f,\fp_0}(\tilde{\sigma} \Frob_\fq \tilde{\sigma}^{-1}).$$
(We note that these equalities are independent of the lift $\tilde{\sigma}$.)
By Equation~\eqref{E:congs}, we have 
\[
\Tr(\rhobar_{f,\fp_0}^\sigma(\Frob_\fq)) = a_{\fq^\sigma}(f) \equiv a_{\fq}(f) \equiv \Tr(\rhobar_{f,\fp_0}(\Frob_\fq)) \pmod{\fp_0},
\]
for all primes $\fq$. Since $\rhobar_{f,\fp_0}$ is irreducible (hence absolutely irreducible since $\fp_0 \nmid 2$), the representations $\rhobar_{f, \fp_0}$ and $\rhobar_{f,\fp_0}^\sigma$ are isomorphic. 
From \cite[Theorem 2.14, Chapter III]{feit}, we conclude that there is an irreducible representation $\rhobar : G_\Q \to \GL_2(\overline{\F}_7)$  such that $\rhobar|_{G_K} = \rhobar_{f,\fp_0}$.

We will now determine the Serre level $N = N(\rhobar)$, the nebentypus $\epsilon(\rhobar)$ and 
Serre weight $k(\rhobar)$. 

We have $k(\rhobar) = 2$ because $\rhobar|_{G_K} = \rhobar_{f,\fp_0}$ is of parallel weight~$2$ and 7 is unramified in~$K$.

Recall that $2$ and $3$ are inert in~$K$, and that $13$ is the only ramified prime in $K$.
By assumption $\rhobar_{f,\fp_0}$ ramifies at
all primes $\fq \mid \fN = 2 \cdot 3 \cdot \fq_{13}$, 
hence $\rhobar_{f,\fp_0}$ is a Steinberg representation at 
these primes (since $\fq \Vert \fN$). It follows that $\rhobar$ 
is Steinberg at 2 and 3; thus $\vv_{\ell}(N(\rhobar)) = 1$ for $\ell =2,3$. 
At~$\ell = 13$, the representation $\rhobar$ is either Steinberg or a twist of Steinberg by a character $\chi$ of~$G_\Q$ 
that becomes trivial over $K$. In the latter case, $\chi$ is of conductor $13^1$ and the conductor of $\rhobar$ at~$13$ is $13^2$. 
So twisting $\rhobar$ by $\chi$, if necessary, we can assume also that $\vv_{13}(N(\rhobar)) = 1$ since both $\rhobar$ and 
$\rhobar \otimes \chi$ restrict to $\rhobar_{f,\fp_0}$. Thus $N(\rhobar) = 2 \cdot 3 \cdot 13 = 78$.

Since $\rhobar$ is Steinberg at $\ell = 2,3,13$ it follows that 
the nebentypus $\epsilon(\rhobar)$ is locally trivial at~$\ell$;
thus $\epsilon(\rhobar) = 1$ as there are no other ramified primes for $\rhobar$.

From the above discussion and Khare-Wintenberger (Serre's conjecture)~\cite{kw09}, 
$\rhobar$ is modular~; hence, there is a classical newform~$h$ and a prime $\fp \mid 7$ in $\Q_h$ such that $\rhobar \simeq \rhobar_{h, \fp}$.  it follows that there is a newform $h \in S_2(78)^{\rm new}$ such that $\rhobar \simeq \rhobar_{h,\fp}$ for 
some prime $\fp \mid 7$ in~$\Q_h$.

There is only one such $h$, 
corresponding to the isogeny class of the elliptic curve $W$, with Cremona label $78a1$,
whose trace of Frobenius at $q=5$ is $a_5(W) = 2$. The prime $5$ splits
in $K$ so, for $\fq \mid 5$, we must have the following congruences
\[
 \Tr(\rhobar_{f,\fp_0}(\Frob_\fq)) = \Tr(\rhobar(\Frob_\fq)) = \Tr(\rhobar_{h,\fp}(\Frob_\fq)) = \Tr(\rhobar_{W,7}(\Frob_5)) \equiv 2 \pmod{7}.
\]
But we easily check that we have 
$\rhobar_{f,\fp_0}(\Frob_\fq) \not\equiv 2 \pmod 7$. 

This contradiction implies that $\rhobar_{f,\fp_0}$ must be reducible, as desired.

\end{proof}

\begin{remark} Let $f$ be a newform in the unique Hecke constituent $S_{18}$ such that $7$ ramifies in~$\Q_f$
(as described in Table~\ref{table:coeff-fields}). Then, one easily checks that $\rhobar_{f, \fp}$ is irreducible for the unique prime 
$\fp \mid \mathfrak{f}_{\Q_f}$ above $7$. By adapting the proof of Proposition~\ref{P:reducible}, one can show that 
$f$ is in fact congruent to the base change of one of the newforms in  $S_2(78, \chi)^{\rm new}$. In this case, we see
that $\rhobar_{f, \fp}$ is unramified at $\fq_{13}$. 
\end{remark}

\begin{remark} Let $f \in S_2(\fN)^{\rm new}$ be a newform as in Proposition~ \ref{P:reducible}. 
Then, $f$ appears in the cohomology of a Shimura curve. For
$\fq = (2), (3)$ and $\fq_{13}$, we have $U_\fq f = \pm f$, where $U_\fq$ is the Hecke operator at $\fq$.
We believe that the reducibility of $\rhobar_{f, \fp}$ can be proved by generalising work of Martin~\cite{martin} or 
of Ribet and Yoo~\cite{ribet10, yoo} on non-optimal levels for reducible representations to the setting of Shimura 
curves over totally real number fields.
\end{remark}

\section{Reduction to the problem of eliminating the form~$g$}
\label{S:redg}

We recall that in order to get a complete resolution of equation~\eqref{E:13137}, it only remains to
deal with the prime $p = 7$. There are five possible obstructing forms to this, four of which are 
described in Proposition~\ref{P:reducible} and Table~\ref{table:coeff-fields};
and the form $g$ defined over $\Q(\sqrt{13})$ described in the introduction. 
The goal of this section is to explain how we discard the four forms in Proposition~\ref{P:reducible} by
using the refined elimination technique described in Section~\ref{S:elimination}. We then deal with~$g$ in three different 
ways in Sections~\ref{S:local}, \ref{sec:global-proof}, and \ref{sec:descent-unit}.

Since we are following the same strategy as in \cite{BCDF}, a crucial step in solving~\eqref{E:13137}
for $p = 7$ is to extend~\cite[Theorem 7]{BCDF} to the exponent $p=7$. To this end, we must prove 
the following:

\begin{theorem}\label{thm:discard-g}
Let~$(a,b,c)$ be a non-trivial primitive solution to equation~\eqref{E:13137} for $p = 7$. 

Then $13 \mid a+b$ and $4 \mid a+b$.
\label{T:thm7}
\end{theorem}

From~\cite[Remark~7.4]{BCDF}, Theorem~\ref{thm:discard-g} follows provided we can discard the form~$g$.
The following theorem reduces the resolution of~\eqref{E:13137} for $p=7$ to the elimination of the form $g$
by discarding the other four obstructing forms. 

\begin{theorem} Assume Theorem~\ref{T:thm7}. Then Theorem~\ref{T:main} holds.
\label{T:reductionToThm7}
\end{theorem}

\begin{proof}[Proof of Theorem~\ref{T:reductionToThm7}]
It suffices to consider $x^{13} + y^{13} = 3 z^n$ with $n = p$ a prime number.

The case $p \not= 7$ is precisely \cite[Theorem 2]{BCDF}, so we can assume $p=7$.

Suppose~$(a,b,c)$ is a non-trivial primitive solution to equation~\eqref{E:13137} with $p = 7$. 
From Theorem~\ref{T:thm7} we can assume that $4 \mid a+b$ and $ 13 \mid a+b$. 

Let $F = F_{a,b}$ be the Frey curve defined over $K$, which was introduced in \cite[\S 7.2]{BCDF}.
Note that $\vv_2(a+b)=\vv_2(3c^p)\ge3$ and $3 \mid a+b$, therefore \cite[Lemma~11]{BCDF}  
gives that $\rhobar_{F,7} \cong \rhobar_{f,\fp}$, 
for some $f \in S_2(\fN)^{\rm new}$ and $\fp \mid 7$ in $\Q_f$.

The elimination technique described in Section~\ref{S:elimination}, applied with the quantities $A_q(f)$ for the rational primes $q=5,7,11,17,31$, 
shows that $\rhobar_{F,7} \not\cong \rhobar_{f,\fp}$ except when $f$ is as in Proposition~\ref{P:reducible}. To complete the proof of the theorem, 
we make use of the refinement explained in Section~\ref{S:elimination} to rule out the remaining four forms. To this end, let $f$ be as in 
Proposition~\ref{P:reducible} (see also Table~\ref{table:coeff-fields}),  and suppose that $\rhobar_{F,7} \cong \rhobar_{f,\fp}$, where $\fp \mid 7$ in $\Q_f$. 
Since $\rhobar_{F,7}$ has conductor $\fN$, the same is true for $\rhobar_{f,\fp}$. Furthermore, $\rhobar_{f,\fp}$
is irreducible (because $\rhobar_{F,7}$ is irreducible by \cite[Theorem~8]{BCDF}) so
by Proposition~\ref{P:reducible} we can assume 
that $\fp$ is unramified in $\Q_f$.
Choosing a $q\not=2,3,13$ satisfying $q\not\equiv1\pmod{13}$, we obtain from \cite[Lemma~8]{BCDF} that
\begin{enumerate}[(i)]
\item either $q\nmid a+b$ and then for all~$\fq$ above~$q$, we have $a_{\fq}(f)\equiv a_{\fq}(F_{a,b})\pmod{\fp}$;
\item or $q\mid a+b$ and then for all~$\fq$ above~$q$, we have $a_{\fq}(f)\equiv \pm(\Norm(\fq)+1)\pmod{\fp}$.
\end{enumerate}
By computing~$a_{\fq}(F_{x,y})$ for each~$\fq\mid q$ and all~$x,y\in\{0,\dots,q-1\}$ not both zero, we eliminate~$f$ by checking that neither of the above 
congruences hold for all unramified primes $\fp \mid 7$ in $\Q_f$. Indeed, the auxiliary prime $q=5$ suffices to eliminate all the four possible~$f$.
\end{proof}

\section{An image of inertia argument with an abelian surface}
\label{S:local}

In this section, we give a proof of Theorem~\ref{T:thm7} based on generalizing an image of inertia argument to the setting of abelian surfaces.

From now on, we let $K = \Q(w)$ where $w^2 = 13$. We let $\calO_K = \Z[u]$, with $u = \frac{1 + w}{2}$,
be the ring of integers of~$K$. We consider the hyperelliptic curve $C$ defined over~$K$ by 
\begin{align}\label{eq:curve}
  C:\, y^2 & = (32 u + 36) x^6 + (24 u + 40) x^5 + (-u - 32) x^4 + (-16 u + 8) x^3 + (17 u - 28) x^2\\
  &\qquad{} + (-6 u + 16) x + 6 u - 16,\nonumber
\end{align}
and we denote its Jacobian by $J$. 

\begin{lemma}\label{lem:cond} The surface $J$ has potentially good reduction at the primes $(2)$ and $(w)$, and we
have $\mathrm{cond}(J) = \fN^2,$ where $\fN = (2^3 w^2)$.
\end{lemma}

\begin{proof} We use \verb|Magma| \cite{magma} to compute the odd part of the conductor of $J$, and
the Dokchitser-Doris algorithm~\cite{DD17} to get the even part. This yields that $\mathrm{cond}(J) = \fN^2$.

Let $\fq$ be either of the primes $(2)$ and $(w)$ of $K$ and consider $C$ over $K_\fq$. By \cite[Th\'eor\`eme 1 (V)]{liu93}, there is a stable model $\mathscr{C}$ of $C$ over an extension $F$ of $K_\fq$ such that the special fiber of $\mathscr{C}$ is a union of two elliptic curves intersecting at a single point.

Let $\mathscr{J}$ be the N\'eron model of the Jacobian of $\mathscr{C}$. By the discussion preceding \cite[Proposition 2]{liu93} and \cite[Proposition 2(v)]{liu93} itself, the special fibre of $\mathscr{J}$ is an abelian variety. Hence, $J$ has potentially good reduction at the primes $(2)$ and $(w)$.
\end{proof}

We record the following additional lemma for later, which also proves more concretely the assertion that $J$ has potentially good reduction at $(w)$. 
Let $K' = \Q(\zeta_{13} + \zeta_{13}^{-1})$ be the maximal totally real subfield of $\Q(\zeta_{13})$. Then $K'/K$ is a cyclic extension 
ramified at $(w)$ only. Let $K_w'$ be the completion of $K'$ at the unique prime above $(w)$. 

\begin{lemma}
\label{lem:good13}The surface $J$ acquires good reduction over $K_{w}'$.
\end{lemma}
\begin{proof}
Using \verb|Magma| \cite{magma}, the conductor exponent of $J$ at the unique prime of $K'$ above $w$ is computed to be $0$.
\end{proof}

Let $g$ be the Hilbert newform over $K$ with parallel weight $2$, trivial character and level $\fN = (2^3w^2)$
listed in \cite[Proposition~9]{BCDF}. Since $\mathrm{ord}_2(\fN) = 3$ is odd, the local component $\pi_{g,2}$
of the automorphic representation $\pi_g$ attached to $g$ is supercuspidal. Therefore, the Eichler-Shimura conjecture
for totally real fields holds (see~\cite[Proposition 2.20.2]{nek12} or~\cite[Theorem B]{zha01}). Thus, there is an abelian
surface $A_g$ with RM by $\Z[\sqrt{2}]$ attached to $g$. The theorem below shows that $A_g$ is isogenous to $J$.

\begin{theorem} \label{T:modularity}
Let $C$ and $J$ be given by \eqref{eq:curve}. Then, we have the following:
\begin{enumerate}
\item\label{item:1} The ring $\End_{K}(J)$ contains~$\Z[\sqrt{2}]$, i.e.\ $J$ is of $\GL_2$-type with real multiplication (RM) by $\Q(\sqrt{2})$;
\item\label{item:2} The surface~$J$ is modular and corresponds to the Hecke constituent of the Hilbert newform $g$. In other words,
$J$ and $A_g$ are isogenous. 
\end{enumerate}
\end{theorem}

\begin{proof} In \cite[Theorem 17]{ek14}, there is an equation for the Humbert surface $\mathcal{H}_8$ of discriminant~$8$,
which parametrizes principally polarized abelian surfaces with RM by $\Z[\sqrt{2}]$ (but where the action of $\Z[\sqrt{2}]$ by endomorphisms is forgotten), and is birational to the projective plane~$\mathbf{P}^2_{r,s}$.
 
Let $Y_{-}(8)$ be the Hilbert modular surface in \cite{ek14} which is the coarse moduli space which parametrizes principally polarized abelian surfaces with real multiplication by $\Z[\sqrt{2}]$. The loc.\ cit.\ gives the Hilbert modular surface $Y_{-}(8)$ as a double-cover of $\mathbf{P}_{r,s}^2$ and this is a birational model of $Y_{-}(8)$ over $\Q$. Additionally, a birational map from $Y_{-}(8)$ to the moduli space $\mathcal{M}_2$ of genus $2$ curves is described. When this map is defined and evaluated at a point $(A,\phi)$, where $\phi$ is the moduli structure imposed on $A$ by $Y_{-}(8)$,  it gives the Igusa-Clebsch invariants of the genus $2$ curve $C'$ whose Jacobian is $A$.

The point  
$$ (r', s') := \left(\frac{-20u - 24}{81}, \frac{-1456u + 3354}{81}\right)$$
determines the $\overline{K}$-isomorphism class of an abelian surface $A$ arising as the Jacobian of $C'$ whose Igusa-Clebsch invariants are given by
\begin{align*}
I_2' & =  \frac{18112u - 38832}{81},\\ 
I_4' & = \frac{-112736u + 270660}{6561},\\ 
I_6' & = \frac{2386589920u - 5484934104}{531441},\\ 
I_{10}' & = \frac{532320256u - 1222121472}{3486784401}.
\end{align*}
Using {\tt Magma}'s implementation of Mestre's method \cite{mestre}, we can find a hyperelliptic curve $C'$ of genus $2$ defined over $K$ whose Igusa-Clebsch invariants are the $I_{2i}'$ above for $i = 1, 2, 3, 5$.  Hence, $A$ arises from the Jacobian of $C'/K$ and is thus itself defined over $K$.

Let $I_2, I_4, I_6$ and $I_{10}$ be the Igusa-Clebsch invariants of the curve $C$, and $\alpha = -60u- 48$. Then,
we have $I_{2i} = \alpha^{2i}I_{2i}'$ for $i = 1, 2, 3, 5$. This shows that $C/K$ is isomorphic to $C'/K$ over $\overline{K}$.

The points in $Y_{-}(8)$ corresponding to $(r',s')$ are still rational over $K$, hence by the modular interpretation of $Y_{-}(8)$ there is a choice of $A$ in its $\overline{K}$-isomorphism class such that $A$ is defined over $K$ and $A$ has RM by $\Z[\sqrt{2}]$. Since $\Aut(C) \cong \Aut(C')$ is of order $2$, we see $C/K$ and $C'/K$ differ by a quadratic twist. Thus, $\End_K(J)$ also contains $\Z[\sqrt{2}]$ by \cite[Lemma 2.2]{Shnidman}, hence proving~(\ref{item:1}).

Alternatively, the algorithm in \cite{cmsv} can be used to show the endomorphism ring $\End_K(J)$ contains $\Z[\sqrt{2}]$. We thank  N.\ Mascot and the authors of \cite{cmsv} for checking this.

We now prove (\ref{item:2}). Recall that $3$ is inert in $\Q(\sqrt{2})$, and consider the $2$-dimensional 
$3$-adic Galois representation attached to $J$
$$\rho_{J, 3}:\,\Gal(\Qbar/K) \to \GL_2(\Q_3(\sqrt{2})).$$ 
Since $J$ is of $\GL_2$-type, it follows from Ribet \cite[Proposition 3.4]{ribet} (note that the proof generalizes to abelian varieties defined over totally real fields) that the determinant of~\(\rho_{J,3}\) is the cyclotomic character, hence for $\fq \nmid 3\fN$ its characteristic polynomial at~$\fq$ is of the form
$$\mathrm{charpoly}(\rho_{J, 3}(\Frob_\fq)) = x^2 - a_\fq x + \Norm(\fq),$$
where $\tr(\rho_{J, 3}(\Frob_{\fq})) = a_\fq \in \Q(\sqrt{2})$, and $\Norm(\fq)$ is the norm of $\fq$. 

Let $\bar{\rho}_{J, 3}$
be the mod $3$ reduction of this representation. By construction, $\bar{\rho}_{J, 3}$ is odd and we see that the image of $\bar{\rho}_{J, 3}$ 
is contained in 
$$\left\{ u \in \GL_2(\F_9) : \det(u) \in \F_3^\times \right\}$$ and, 
since $-1$ is a square in $\F_9^\times$, the projective image lands in $\PSL_2(\F_9)$.
By computing the orders of the conjugacy classes
of $\bar{\rho}_{J, 3}(\Frob_\fq)$ for the primes $\fq$ above $17$ and $53$, we see that the
projective image of $\bar{\rho}_{J, 3}$ contains elements of orders $2, 4$ and $5$. There
is no proper subgroup of $\PSL_2(\F_9)$ which contains three elements with those orders,
hence the projective image of $\bar{\rho}_{J, 3}$ is $\PSL_2(\F_9)$. In particular, we see that 
the image of $\bar{\rho}_{J, 3}$ contains $\SL_2(\F_9)$, so $\bar{\rho}_{J,3}$  is absolutely 
irreducible. 

The prime $3$ splits in $K$. Writing  $(3) = v_1 v_2$, where $v_1 = (u - 1)$ and $v_2 = (u)$, we get that
 $$\tr(\rho_{J, 3}(\Frob_{v_1})) = 2 \pm \sqrt{2},\,\,\text{and}\,\, \tr(\rho_{J,3}(\Frob_{v_2})) = \pm \sqrt{2}$$
 by computing the Euler factors of the curve $C$ at both places and factoring over $\Q(\sqrt{2})$. 
 These traces are units modulo $3$, so $\rho_{J,3}$ is ordinary at each $v \mid 3$. Further, since $5 \nmid \fN$,
 we see that $\bar{\rho}_{J, 3}|_{I_5}$ is trivial, hence it has odd order. Since $3$ and $5$ have odd 
 ramification indices in $K$, it follows that
 $\bar{\rho}_{J, 3}$ satisfies the conditions of \cite[Theorem 3.2 and Proposition 3.4]{ell05}.
 Hence, it is modular. 
 
We use~\cite[Theorem 3.5.5]{kis09} to conclude that $\rho_{J, 3}$, and hence $J$, is modular. 

By local-global compatibility (\cite[Th\'eor\`eme (A)]{car86}) and Lemma~\ref{lem:cond}, the level of the Hilbert newform 
attached to $J$ is $\fN = (2^3 w^2)$.  There is a unique Hecke constituent of weight $2$ and level $\fN$
whose Euler factors match those of the surface $J$, it is the one corresponding to the newform $g$. 
\end{proof}

The `image of inertia argument' discards the possibility of the mod~$p$ representation of the Frey curve 
and that of a newform $f$ being isomorphic by showing they have different image sizes 
at an inertia subgroup. This idea, originally from \cite[p. 8]{Kraus1998}, has been 
extensively applied and refined~\cite{BennetSkinner}  (see \cite[\S 3]{BCDF} for a description of two refinements) 
in the case of $f$ corresponding to an elliptic curve.
In its essence, this argument boils down to showing that the Frey curve 
and the newform $f$ have different inertial types at some prime~$q$ 
dividing the level of $f$. Thus far, such inertia arguments have been restricted to the case of $f$ corresponding to an elliptic curve because a method to explicitly determine the inertial types of a form~$f$ with non-rational coefficients has not been worked out in general.

In this section, we will use the modularity in Theorem~\ref{T:modularity} to describe  (see Theorem~\ref{T:localIsos}) the inertial type of the non-rational form~$g$ at the prime $(2)$; this together with the 
local information at $(w)$ given by Lemma~\ref{lem:good13} allows for a proof of Theorem~\ref{T:main}.

Before proceeding we need some notation. 
Let $E = E_{1,-1}$ be the Frey curve attached to the trivial solution $(1,-1,0)$ in \cite[Section~7.1]{BCDF}; 
it admits a minimal model given by
$$E:\, y^2 = x^3 - u x^2 + (9u - 25)x - 17u + 49.$$
Let also $J$ be as above and 
$\fp_7$ be the prime of $\Q(\sqrt{2})$ above $7$ generated by $3 + \sqrt{2}$.

Let $K_2/\Q_2$ be the unique unramified quadratic extension
and $K_2^{un}$ its maximal unramified extension in a fixed algebraic closure of~$\Q_2$.
For an abelian variety $A/K_2$ with potentially good 
reduction, there is a minimal extension $M_A/K_2^{un}$ 
where $A$ obtains good reduction. By a result of Serre-Tate
(\cite[\S 2, Corollary~3]{ST1968}), we have 
$M_A = K_2^{un}(A[p])$ for any odd prime $p$. 
 
We recall that the curve $E$ has potentially good reduction at $2$.
By Lemma~\ref{lem:cond}, the same is true for $J$; so
$M_E = K_2^{un}(E[3])$ and~$M_J = K_2^{un}(J[3])$. 

\begin{proposition}\label{P:3torsion} 
We have $M_E = M_J$ and $\Gal(M_E/K_2^{un}) \cong \SL_2(\F_3)$.  
\end{proposition}

\begin{proof} Computing the standard invariants of $E$, we find that they have the following valuations at $(2)$:
\[
(\vv_2(c_4(E)),\vv_2(c_6(E)),\vv_2(\Delta(E)))=(5,5,4).
\]
Hence $\vv_2(j(E))=11$ and~$E$ has potentially good reduction at~$(2)$. It follows from \cite[pp. 675, Corollaire]{Cali} that $E$ has semistability defect $e=24$, hence 
$\Gal(M_E/K_2^{un}) \cong \SL_2(\F_3)$ by~\cite{Kraus1990}.

On the other hand, $J$ also has potentially good reduction at~$(2)$ by Lemma~\ref{lem:cond}, and as a byproduct, the Dokchitser-Doris algorithm~\cite{DD17} returns 
the totally ramified field $K_2(J[3])$ with a defining polynomial over $K_2$ of degree $24$.

We check that $E$ has good reduction over $K_2(J[3])$, so $M_E \subset M_J = K_2^{un}(J[3])$ 
by minimality. Since $[M_E : K_2^{un}] = 24 = [M_J : K_2^{un}]$ the result follows.
\end{proof}

The following well known result will be of use for us; due to a lack of a clear reference 
we include a proof here.

\begin{lemma} \label{L:exceptionalInertia} Let $K$ be a totally real field 
and $\fq$ a prime above 2 in $K$. Let $I_{\fq} \subset G_K$ be an inertia subgroup at ${\fq}$. 
Let $h$ be a Hilbert modular form over $K$ of level $\fN$ and field of coefficients~$\Q_h$. 
Assume that $\fq \mid \fN$ and that $h$ has a supercuspidal exceptional type at $\fq$.

Then, for all primes $p$ coprime to $6\fN$ and  all primes $\fP \mid p$ in $\Q_h$, we have
$\rhobar_{h,\fP}(I_{\fq}) \simeq \SL_2(\F_3)$. 
\end{lemma}
\begin{proof}
Let $\pi_h$ be the automorphic representation attached to $h$, and $\pi_{h, \fq}$ the local component at~$\fq$. Also, let $\sigma_{h, \fq}: W_{\fq} \to \GL_2(\C)$ 
be the Weil representation attached to $\pi_{h, \fq}$ by the local Langlands correspondence (see~\cite{kut80}). Since $\pi_{h, \fq}$ is a supercuspidal exceptional representation, then $\sigma_{h, \fq}$ is an exceptional representation, which means that the projective 
image of $\sigma_{h,\fq}$ is either $A_4$ or $S_4$ (the $A_5$ case cannot occur since $W_{\fq}$ is solvable).

Let $p$ be a rational prime coprime to $6\fN$ and $\fP\mid p$ a prime in $\Q_h$. Let $D_{\fq}\supset I_{\fq}$ be a decomposition group at~$\fq$ in~$G_K$. By local-global compatibility (\cite[Th\'eor\`eme (A)]{car86}), the projective image of $\rho_{h, \fP}|_{D_{\fq}}$ in $\PGL_2(\overline{\Q}_p)$ is either $A_4$ or $S_4$, and $\rho_{h, \fP}|_{I_{\fq}}$ acts irreducibly  \cite[\S 42.1]{BH06} (in loc.\ cit., supercuspidal exceptional representations are called primitive representations). Since $p \nmid 6$, the image of $\rhobar_{h, \fP}|_{D_{\fq}}$ in $\PGL_2(\overline{\F}_p)$ is also $A_4$ or $S_4$, and $\rhobar_{h, \fP}|_{I_{\fq}}$ acts irreducibly. A careful analysis of the proof of~\cite[Proposition 2.4]{dia95} shows that it carries over to any finite local extension of $\Q_2$.
In particular, this implies that the projective image of $\rhobar_{h, \fP}|_{I_{\fq}}$ is equal to $A_4$. Therefore, the 
image of $\rhobar_{h, \fP}|_{I_{\fq}} \subset \SL_2(\Fbar_p)$ is isomorphic to either $A_4$ or $\SL_2(\F_3)$; the result 
now follows since there is only one element of order $2$ in $\SL_2(\overline{\F}_p)$ (for $p > 2$), hence  
no subgroup of $\SL_2(\overline{\F}_p)$ is isomorphic to $A_4$.  
\end{proof}

\begin{theorem} \label{T:localIsos}
Let $\mathfrak{P} \mid 7$ in $F = \Q(\sqrt{2})$ be a prime. Then, we have
$$M_J = K_2^{un}(J[7]) = K_2^{un}(J[\mathfrak{P}])$$
and, moreover, 
\[
 \rhobar_{g,\mathfrak{P}}|_{I_2} \cong \rhobar_{E,7} |_{I_2}.
\]
Here, $I_{2}$ denotes an inertia subgroup at $(2)$ in~$G_K$.
\end{theorem}
\begin{proof} Let $\mathfrak{P}, \mathfrak{P}'$ be the two primes above $7$ in $\Q(\sqrt{2})$. By \cite[Theorem 4.3.1]{wilson98}, we have that
$$J[7] = J[\mathfrak{P}] \times J[\mathfrak{P}'].$$
This means that the field $M_J = K_2^{un}(J[7])$ is the compositum of $K_2^{un}(J[\mathfrak{P}])$ and $K_2^{un}(J[\mathfrak{P}'])$, the fields cut out by $\rhobar_{J, \mathfrak{P}}|_{I_2}$ and $\rhobar_{J, \mathfrak{P}'}|_{I_2}$ respectively. 

We will now show that these three fields have the same degree and the first statement follows. 
By Theorem~\ref{T:modularity}, we have
$\rhobar_{J, \mathfrak{P}} \cong \rhobar_{g, \mathfrak{P}}$ and $\rhobar_{J, \mathfrak{P}'} \cong \rhobar_{g, \mathfrak{P}'}$. By Proposition~\ref{P:3torsion} and the discussion preceding it, we only need
to show that the fields cut out by $\rhobar_{g, \mathfrak{P}}|_{I_2}$ and $\rhobar_{g, \mathfrak{P}'}|_{I_2}$ have degree $\# \SL_2(\F_3) = 24$. 

Note that if $\rho_{g,\mathfrak{P}}|_{I_2}$ is reducible then the conductor exponent at $2$ is either 1 (special representation) or even (because the determinant of $\rho_{g,\mathfrak{P}}$ is cyclotomic, and hence on restriction to inertia, the diagonal characters must be inverses of each other); therefore, 
$\rho_{g,\mathfrak{P}}|_{I_2}$ is irreducible and $g$ has supercuspidal type at $2$ which is not given by an induction from the unramified quadratic extension. We conclude that $g$ is either supercuspidal induced from a ramified extension or exceptional. In the former case, then $\rhobar_{g,\mathfrak{P}}|_{I_2}$ would have projective dihedral image; since the field cut out by $\rhobar_{g,\mathfrak{P}}|_{I_2}$ is a Galois subextension of $K_2^{un}(J[7])/K_2^{un}$, which has Galois group $\SL_2(\F_3)$, and $\SL_2(\F_3)$ does not have any quotients which are dihedral, we conclude that $g$ must have a supercuspidal exceptional type at $2$. By Lemma~\ref{L:exceptionalInertia}, we obtain that the fields cut out by $\rhobar_{g,\mathfrak{P}}|_{I_2}$ have degree $\# \SL_2(\F_3) = 24$ as required.

We now prove the last statement. From the first part of the theorem, 
Theorem~\ref{T:modularity} and Proposition~\ref{P:3torsion}, we have that $M_J$
is the field cut out by both $\rhobar_{g,\mathfrak{P}}|_{I_2}$ and
$\rhobar_{E,7}|_{I_2}$. Thus, $\rhobar_{g,\mathfrak{P}}|_{I_2}$,
and $\rhobar_{E,7}|_{I_2}$ have the same kernel, and image $\SL_2(\F_3) \hookrightarrow \SL_2(\F_7) \subset \GL_2(\F_7)$. 
Therefore, it follows from \cite[Lemma~2]{F33p} that they are isomorphic representations.
\end{proof}

Finally, we now show Theorem~\ref{T:thm7} using the information on the inertial types.

\begin{proof}[First proof of Theorem~\ref{T:thm7}]

Let~$(a,b,c)$ be a non-trivial primitive solution to equation~\eqref{E:13137} with $n = p = 7$. 

Let $E_{a,b}$ be the Frey curve defined in \cite[Section~7.1]{BCDF}.
The proof of \cite[Theorem 7]{BCDF} uses \cite[Proposition 9]{BCDF} which asserts that 
$\rhobar_{E_{a,b},p} \cong \rhobar_{Z,p}$ where $Z$ is one of $E_{1,-1}$, $E_{1,0}$, $E_{1,1}$. In the case $p = 7$, there is the additional possibility that $\rhobar_{E_{a,b},7} \cong \rhobar_{g,\mathfrak{P}}$ where $\mathfrak{P} \mid 7$ in $\Q(\sqrt{2})$. 
By \cite[Remark 7.4]{BCDF} we have $\mathfrak{P} = \fp_7$. 

Suppose $Z$ is one of the three curves above, the arguments of \cite[Theorem 7]{BCDF} still hold for~$p=7$. 
For instance, if $Z = E_{1,-1}$ and $4 \nmid a + b$, then it is shown that 
\begin{equation}
\label{E:isom}
  \rhobar_{E_{a,b},p} \mid_{I_2} \not \cong \rhobar_{Z,p} \mid_{I_2}. 
\end{equation}
To complete the proof, we only
have to eliminate the possibility that $\rhobar_{E_{a,b},7} \cong \rhobar_{g,\fp_7}$.

By Theorem~\ref{T:modularity}, we have $\rhobar_{g,\fp_7} \cong \rhobar_{J,\fp_7}$. The proof of \cite[Theorem~7~(A)]{BCDF} then remains valid by simply replacing $Z/K'$ with $J/K_w'$ and using Lemma \ref{lem:good13}.
This shows $13 \mid a+b$.

Finally, we note that ${\rhobar_{g,\fp_7}} {\mid_{I_2}} \cong {\rhobar_{E_{1,-1},7}} {\mid_{I_2}}$ by Theorem~\ref{T:localIsos}, and that  ${\rhobar_{E_{1,-1},7}} {\mid_{I_2}} \not\cong {\rhobar_{E_{a,b},7}} {\mid_{I_2}}$ by \eqref{E:isom}, for $4 \nmid a + b$. This shows that $4 \mid a+b$.
\end{proof}

\begin{remark} The obstruction to solving~\eqref{fte:13-13-7} comes from an abelian surface with real multiplication; namely, the surface $A_g$ attached to the form $g$. The approach in this section relies crucially on the fact that $A_g$ is isogenous over $\Q(\sqrt{13})$ to a principally polarized abelian surface with real multiplication, i.e. $A_g$ is a $\Q(\sqrt{13})$-rational point on some Humbert surface. Hence, the methods in~\cite[\S 4.1.2]{dk16} and~\cite[\S 7]{ek14} can be applied to explicitly find a hyperelliptic curve $C$ such that $A_g$ is isogenous over $\Q(\sqrt{13})$ to the Jacobian $J$ of $C$. 
There is thus a reasonable chance for the method in this section to succeed whenever the obstruction to the modular method for solving a Diophantine equation like the one in~\eqref{fte:13-13-7} is isogenous to a principally polarized abelian surface with real multiplication and reasonably-sized height.
\end{remark}

\begin{remark}
Although we apply this generalized inertia argument to a situation with fixed exponent $p = 7$, the method described is also applicable to a setting with general exponent~$p$ (unlike the methods in Sections~\ref{sec:global-proof}-\ref{sec:descent-unit}). This can be useful when working with certain Frey hyperelliptic curves as in the Darmon program for the Generalized Fermat equation~\cite{DarmonDuke}.

For example, in forthcoming work \cite{CK}, the equation $x^p + y^p = z^5$ is studied using a Frey hyperelliptic curve $C_5^- = C_5^-(a,b,c)$ 
constructed by Darmon (see~\cite[p. 425]{DarmonDuke}), where $(a,b,c)$ is a non-trivial primitive solution satisfying $2 \mid ab$, $5 \nmid ab$. 
The Jacobian $J_5^- = J_5^-(a,b,c)$ of~$C_5^-$ is defined over~$\Q$ and becomes of~$\GL_2$-type over~$F = \Q(\sqrt{5})$ with real multiplication by~$F$. In applying the modular method in this case, 
one finds that it is not possible to rule out an isomorphism 
\begin{equation}
\label{cong-iso}
  \rhobar_{J_5^-,\fP} \cong \rhobar_{h,\fP}
\end{equation}
by comparing traces of Frobenius at primes not dividing $10$,  where $h$ is a certain Hilbert newform of parallel weight $2$, trivial character, and level $(2^4 \mathfrak{r}^2)$ over $F$; here, $\fP$ is a prime above $p$, $\mathfrak{r}$ is the prime above $5$ of $F$, and note that~$2$ is inert in~$F$. By modularity and comparing a few traces, one checks that the Hilbert newform $h$ corresponds to $J_5^-(8,-8,0)$ twisted by the quadratic character associated to $F(\sqrt{2})$. Applying the Dokchitser-Doris algorithm~\cite{DD17}, we deduce that $J_5^-(8,-8,0)$ achieves good reduction over a degree $10$ totally ramified extension of $F_{2}^{un}$  (see~\cite{programs}), so $\rhobar_{h,\fP}(I_2)$ has order dividing~$20$ and divisible by $5$. Moreover, it is shown in \cite{CK} that $J_5^-(a,b,c)$ has potentially multiplicative reduction at~$2$ when $2 \mid ab$,  $5 \nmid ab$ for $p > 5$.  Therefore, this implies that~$\rhobar_{J_5^-,\fP}$ is a quadratic twist of Steinberg at~$2$ and hence~$\rhobar_{J_5^-,\fP}(I_2)$ has order dividing~$2p$. We thus obtain a contradiction to isomorphism \eqref{cong-iso} when $p > 5$.
\end{remark}

\section{A residual multiplicity one argument} 
\label{sec:global-proof}

In this section, we outline a proof of Theorem~\ref{T:thm7} based on establishing the mod $7$ congruence mentioned in the introduction and described below; see also ~\cite[Remark 7.4]{BCDF} for more details. 

The natural attempt to use a Sturm bound such as the one in \cite{BP} fails because it necessitates the computation of Hecke eigenvalues for primes of $K = \Q(\sqrt{13})$ of norm at least $10^6$, which is not computationally feasible.

The method we introduce for establishing $\bar{\rho}_{g, \fp_7} \cong \bar{\rho}_{f, 7}$ below, without the use of a Sturm bound, relies on a residual multiplicity one argument. It applies for more general $f$, $g$ and~$p$. For example, it has also been used in~\cite{dem20} to show the isomorphism of some residual mod~$2$ Galois representations arising from Hilbert modular forms over the maximal totally real subfield of $\Q(\zeta_{32})$, the cyclotomic field of $32$nd root of unity.

Let $E = E_{1,-1}$ and $\fp_7$ be as in the previous section; in particular, we have $\F_{\fp_7} = \F_7$. 

Let $S_2(\fN)^{\rm new}$ be the new subspace of Hilbert cusp forms of weight $2$, trivial character, and level $\fN  =(2^3w^2)$.  The elliptic curve $E$ is modular by \cite{FLHS} and corresponds to a newform $f \in S_2(\fN)^{\rm new}$. Similarly, in Theorem~\ref{T:modularity}, we prove that the surface $J$ is modular and corresponds to the form $g \in S_2(\fN)^{\rm new}$. 
So, we have $\rhobar_{f, 7} \simeq \rhobar_{E, 7}$ and $\rhobar_{g, \fp_7} \simeq \rhobar_{J, \fp_7}$.

Let $D = \left(\frac{-1,-1}{K}\right)$ be the totally definite quaternion algebra over $K$ ramified at both places at infinity only. Let $\calO_D$ be a maximal order in $D$. Also, let $\widehat{D} = D\otimes \widehat{\Z}$, $\widehat{\calO}_K=\calO_K\otimes\widehat{\Z}$, and $\widehat{\calO}_D = \calO_D\otimes \widehat{\Z}$, where
$\widehat{\Z}$ is the profinite completion of~$\Z$. We fix an isomorphism $\left(\widehat{\calO}_D\right)^\times \simeq \GL_2(\widehat{\calO}_K)$, and we define the compact open subgroup
\begin{eqnarray*}
U_0(\fN) := \left\{\gamma \in \left(\widehat{\calO}_D\right)^\times: \gamma \equiv \begin{pmatrix} \ast & \ast \\ 0& \ast\end{pmatrix} \mod \fN \right\}.
\end{eqnarray*}

We consider the space
$$S_2^D(\fN) = \left\{f:\,D^\times \backslash \widehat{D}^\times/U_0(\fN) \to \C\right\}$$
which has an action of Hecke operators $T_\fq$ for $\fq \nmid \fN$.

Let $\T_\fN$ be the Hecke algebra acting on~$S_2^D(\fN)$ which is generated by the operators $T_\fq$. There is a decomposition
\begin{equation*}
   S_2^D(\fN) = S_2^D(\fN)^{\rm old} \oplus S_2^D(\fN)^{\rm new}.
\end{equation*}
where the subspaces $S_2^D(\fN)^{\rm old}$ and $S_2^D(\fN)^{\rm new}$ are $\T_\fN$-stable, and $S_2^D(\fN)^{\rm new}$ is the orthogonal complement of $S_2^D(\fN)^{\rm old}$ under a certain inner product defined on $S_2^D(\fN)$.

The  $\Z$-submodule of $S_2^D(\fN)$ given by
$$S_2^D(\fN, \Z) = \left\{f:\,D^\times \backslash \widehat{D}^\times/U_0(\fN) \to \Z\right\},$$
is stable under $\T_\fN$ and generates $S_2^D(\fN)$ over $\C$ (i.e.\ is an integral structure for $S_2^D(\fN)$).

Let $S_2^D(\fN,\Z)^{\rm new}$ be the orthogonal projection of $S_2^D(\fN,\Z)$ to $S_2^D(\fN)^{\rm new}$. Then $S_2^D(\fN,\Z)^{\rm new}$ is an integral structure for $S_2^D(\fN)^{\rm new}$; in particular, the matrices representing the action of 
the $T_\fq$ on $S_2^D(\fN)^{\rm new}$ have
integer coefficients.

For any commutative ring with unity $A$, define
\begin{align*} 
   & S_2^D(\fN,A) := S_2^D(\fN,\Z) \otimes A \\
   & S_2^D(\fN,A)^{\rm new} := S_2^D(\fN,\Z)^{\rm new} \otimes A.
\end{align*}

We recall that, since $[K:\Q] = 2$ is even, the Jacquet-Langlands correspondence implies that there is an isomorphism of Hecke modules $\JL: S_2^D(\fN) \simeq S_2(\fN)$, which maps
$S_2^D(\fN)^{\rm new}$ onto $S_2(\fN)^{\rm new}$. We let $\phi = \JL^{-1}(f)$ and $\psi = \JL^{-1}(g)$.

\begin{proposition} \label{P:mod7Iso} Consider the residual representations
$$\bar{\rho}_{g, \fp_7}:\,\Gal(\Qbar/K) \to \GL_2(\F_7) \quad \text{ and } \quad \bar{\rho}_{f, 7}:\,\Gal(\Qbar/K) \to \GL_2(\F_7).$$
Then, we have $\bar{\rho}_{g, \fp_7} \cong \bar{\rho}_{f, 7}$. 
\end{proposition}

\begin{proof} Recall that the coefficient field of~$f$ is $\Q$ and that of $g$ is $L = \Q(\sqrt{2})$. Let $\calO_L = \Z[\sqrt{2}]$ and $\calO_{L, \fp_7}$ be the completion of $\calO_L$ at
$\fp_7$.  Let
$$\theta:\,S_2^D(\fN, \calO_{L, \fp_7})^{\rm new} \twoheadrightarrow S_2^D(\fN, \F_7)^{\rm new}$$
be the natural reduction map.

Up to scaling, we may assume $\phi$ and $\psi$ are elements in $S_2^D(\fN,\calO_{L,\fp_7})^{\rm new}$ and their reductions $\bar{\phi} = \theta(\phi)$ and $\bar{\psi} = \theta(\psi)$ are non-zero elements in $S_2^D(\fN,\F_7)^{\rm new}$. Hence, $\bar \phi$ and $\bar \psi$ are eigenvectors of $T_\fq \pmod{7}$ for all $\fq \nmid 7 \fN$.

Let $\T \subseteq \End_{\F_7}(S_2^D(\fN, \F_7)^{\rm new})$ be the Hecke algebra generated by the operators $T_\fq \pmod{7}$ for all primes $\fq \nmid 7\fN$, and $W_{\bar{\phi}}$ (resp.\ $W_{\bar{\psi}}$) be the $\T$-submodule generated by $\bar \phi$ (resp.\ $\bar \psi$) in $S_2^D(\fN,\F_7)^{\rm new}$.

Let $\T' \subseteq \T$ be the subalgebra generated by the Hecke operators $T_\fq$ for prime ideals $\fq \nmid 7\fN$ of norm up to $43$, and  $S$ the socle 
of $S_2^D(\fN, \F_7)^{\rm new}$ considered as a $\T'$-module, i.e.\ the sum of all simple $\T'$-submodules of $S_2^D(\fN, \F_7)^{\rm new}$, which is semi-simple.

Since both $\bar{\phi}$ and $\bar{\psi}$ are eigenvectors for $\T'$, $W_{\bar \phi}$ and $W_{\bar \psi}$ are simple $\T'$-modules of dimension one over $\F_7$. 
Hence, $W_{\bar \phi}$ and $W_{\bar \psi}$ are contained in $S$. They are isomorphic as $\T'$-modules as  for all  prime ideals $\fq \nmid 7\fN$ of norm up to $43$, 
we have that $a_\fq(f) \pmod{7} = a_\fq(g) \pmod{\fp_7}$ as elements in $\F_7$, where $a_\fq(u)$ denotes the eigenvalue of $T_\fq$ acting on an eigenvector $u \in S_2^D(\fN)$.

Using \verb|Magma| \cite{magma}, we can compute the $\T'$-module $S$, which has dimension $348$ over $\F_7$. There are $34$ (non-isomorphic) simple constituents 
which have dimension one over $\F_7$, and each appears with multiplicity one. Thus, $W_{\bar \phi} = W_{\bar \psi}$ inside $S \subseteq S_2^D(\fN, \F_7)^{\rm new}$. 
Since $W_{\bar \phi} = W_{\bar \psi}$ inside $S_2^D(\fN, \F_7)^{\rm new}$ and both are $\T$-modules as well, we obtain that $a_\fq(f) \pmod{7} = a_\fq(g) \pmod{\fp_7}$ 
as elements of $\F_7$, for all prime ideals $\fq \nmid 7 \fN$.

\end{proof}

\begin{proof}[Second proof of Theorem~\ref{T:thm7}]
The bulk of the argument of the proof of this theorem given in Section~\ref{S:local}, including the identity \eqref{E:isom}, still applies. 

Therefore, we only need to show that the isomorphism $\rhobar_{E_{a,b},7} \cong \rhobar_{g,\fp_7}$ is not possible.  Suppose we have that $\rhobar_{E_{a,b},7} \cong \rhobar_{g,\fp_7}$. Then by Proposition~\ref{P:mod7Iso}, we have 
\[ \rhobar_{E_{a,b},7} \cong \rhobar_{g,\fp_7} \cong \bar{\rho}_{E_{1,-1}, 7}.\]
Now the arguments that eliminated~$E_{1,-1}$ apply to eliminate $g$, thereby completing the proof.
\end{proof}

\section{A unit sieve argument} 
\label{sec:descent-unit}
In this section, we give a proof of Theorem~\ref{T:thm7} based on combining the study of units in classical descent with
local restrictions on the solutions coming from the multi-Frey approach in~\cite{BCDF}. 

Let $\zeta$ be a primitive 13th root of unity. Suppose $(a,b,c)$ is a nontrivial 
primitive solution to~\eqref{fte:13-13-7}. We then have the factorization in $\Z[\zeta]$,
\begin{align*}
  a^{13} + b^{13} & = (a+b) \prod_{i=1}^{12} (a + \zeta^i b) = 3 c^7.
\end{align*}
The integers \(a + b\) and~\(\frac{a^{13} + b^{13}}{a + b}\) are coprime away from~\(13\) (see, e.g., \cite[Lemma~2.1]{DahmenSiksek}). Let~\(\ell\not\equiv 1\pmod{13}\) be a prime number dividing~\(a^{13} + b^{13}\). Since~\(a\) and~\(b\) are coprime, we have that~\(\ell\nmid ab\) and hence there exists an integer~\(b'\) such that~\(bb'\equiv -1 \pmod{\ell}\).  Therefore we have~\((ab')^{13}\equiv 1\pmod{\ell}\) and since~\(\ell\not\equiv 1\pmod{13}\), it follows that~\(\ell\) divides~\(a + b\). In particular, we have that~\(3\nmid a + \zeta b\). Furthermore, by classical descent, we have that
\begin{equation}
\label{unit-equation}
  a + \zeta b = \begin{cases}
    \epsilon \beta^7  & \text{ if } 13 \nmid a+b \\
    \epsilon (1 - \zeta) \beta^7 & \text{ if } 13 \mid a+b,
   \end{cases}
\end{equation}
where $\epsilon$ is a unit of $\Z[\zeta]$ and $\beta \in \Z[\zeta]$.
We only need to consider $\epsilon$ up to 7th powers, which means that there are initially $16,807$ possible choices for $\epsilon$.

One can now reduce~\eqref{unit-equation} modulo a prime $\fq$ of $\Z[\zeta]$ above the rational 
prime~$q$. If $q$ is such that the order of the 
multiplicative group of the residue field at~$\fq$ is divisible 
by~$7$, the condition of being a 7th power is nontrivial. 

The primes $q = 11, 17, 23, 29, 37, 41$ satisfy the condition above; moreover, we can obtain local information from the modular method: if $\rhobar_{E_{a,b},7} \cong \rhobar_{g,\fp_7}$, locally at~$q$ this congruence imposes
constraints on the solutions $(a,b)$ modulo $q$ (in both cases of good or multiplicative reduction of $E_{a,b}$ at $q$) and hence on the unit $\epsilon$.

Note we cannot obtain information from $q=2$ in the same way,  
because $2$ divides the level of~$g$. However, since 
the multiplicative group of the residual field of $K$ at~$\fq=(2)$ has
order a multiple of 7, this together with the assumption 
$4 \nmid a + b$ 
also imposes restrictions on~$\epsilon$.

\begin{proof}[Third proof of Theorem~\ref{T:thm7}]

We first note the condition $4 \nmid a+b$ 
is equivalent to $2 \nmid a + b$ for our equation.
Now we sieve the set of possible units in the two cases of equation~\eqref{unit-equation}, 
assuming $2 \nmid a+b$ for a contradiction. More precisely, 
in both the cases $13 \nmid a+b$ and $13 \mid a+b$, using the primes
$q = 2, 11, 19, 23$ the set of units which passes all local conditions is empty. 
Thus $4 \mid a+b$, as desired.

We will now prove that $13 \mid a+b$. From the above we can assume
$4 \mid a + b$ and $13 \nmid a + b$ for a contradiction. 
Using the local information at the primes $q = 2, 11, 19, 23, 29, 41$, again the set of units which passes all local conditions is empty, concluding the proof.
\end{proof}

\begin{remark}
This method of sieving units also appears in \cite{DahmenSiksek}, in the context of reducing the number of hyperelliptic curves to be considered for Chabauty. However, only the primes $q$ not dividing the level are considered there. In our case, we succeed without the need to apply Chabauty since we can sieve also at $q=2$ which is a prime in the level.
It should also be noted that the reason the unit $\epsilon = 1$ does not pass the local conditions from the modular method (unlike the situation in \cite{DahmenSiksek}) is because our original equation \eqref{fte:13-13-7} does not have the trivial solutions $\pm (1,0,1), \pm (0,1,1)$ due to the coefficient $3$. 
\end{remark}

\section{A remark on the equation $x^{2\ell} + y^{2m} = z^{p}$}

To conclude this paper, we use the technique from Section~\ref{S:reducible} to 
complete the following result of Anni--Siksek \cite{AnniSiksek}. 
\begin{theorem}[Anni--Siksek]
\label{T:AS}
Let $p=3,5,7,11$ or $13$. Let $\ell, m \geq 5$ be primes, and if $p=13$ suppose moreover that $m, \ell \neq 7$.
Then the only primitive solutions to 
\[
 x^{2\ell} + y^{2m} = z^{p}
\]
are the trivial ones $(x,y,z) = (\pm 1, 0, 1)$ and $(0, \pm 1, 1)$.
\end{theorem}
The proof of this theorem is a remarkable application
of the modular method over totally real subfields of $\Q(\zeta_p)^+$.
The extra conditions $m, \ell \neq 7$ for $p=13$ are required due to the presence of a single 
Hilbert newform, denoted $\ff_{11}$ in {\it loc. cit}, which evades the elimination step. 

Let $K \subset \Q(\zeta_{13})^+$ be the cubic field of Section~\ref{S:reducible}. Recall that $2$ is inert in~$K$ and that $\fq_{13}$ is the unique prime above~$13$. 
The newform $\ff_{11}$ belongs to $S_2(2\cdot\fq_{13})^{\rm new}$, the new subspace of Hilbert cusp forms of level $2 \cdot \fq_{13}$, weight~$2$, and trivial character. 
It is the unique newform in this space with field of coefficients $\Q_{\ff_{11}} = \Q(\zeta_7)^+$ the real cubic subfield of the cyclotomic field of $7$-th roots of unity, in which $7$ is totally ramified. 
Let $\fp_0$ be the unique prime of $\Q_{\ff_{11}}$ above 7. The authors are not able to exclude the possibility that 
\begin{equation} \label{E:f11}
\rhobar_{E',7} \simeq \rhobar_{\ff_{11},\fp_0},
\end{equation}
where $E'$ is the Frey curve defined in \cite[pg. 10]{AnniSiksek}; see \cite[pg. 19]{AnniSiksek} for further details. However, as observed in {\it loc. cit.}, numerical evidence 
strongly suggests that $\rhobar_{\ff_{11},\fp_0}^{ss} \simeq \chi_7 \oplus 1$ so, in particular, $\rhobar_{\ff_{11},\fp_0}$ is reducible. In Proposition~\ref{P:f11}, we show that this is indeed
the case. This means that the isomorphism~\eqref{E:f11} cannot occur since $\rhobar_{E',7}$ is irreducible. Therefore, since $\fp_0$ is the only prime 
in $\Q_{\ff_{11}}$ above 7, we have successfully completed the elimination step for the form $\ff_{11}$. 
This in turn allows us to remove the assumption that $m,\ell \neq 7$ for $p=13$ in Theorem~\ref{T:AS}. Thus we have the following corollary.

\begin{corollary}
Let $\ell, m \geq 5$ be primes. 
Then the only primitive solutions to 
\[
 x^{2\ell} + y^{2m} = z^{13}
\]
are the trivial ones $(x,y,z) = (\pm 1, 0, 1)$ and $(0, \pm 1, 1)$.
\end{corollary}

We now complete the proof of this result by showing that $\rhobar_{\ff_{11}, \fp_0}$ is indeed reducible. 

\begin{proposition} \label{P:f11}
The representation $\rhobar_{\ff_{11},\fp_0} : G_K \to \GL_2(\F_7)$ is reducible.
\end{proposition}
\begin{proof} Suppose $\rhobar_{\ff_{11},\fp_0}$ is irreducible for a contradiction.

The conductor of $\rhobar_{\ff_{11},\fp_0}$ is $\fN' = (1)$, $(2)$, $\fq_{13}$ or $2\fq_{13}$. 
From the irreducibility assumption on~$\rhobar_{\ff_{11},\fp_0}$ 
and \cite[Theorem~3.2.2]{BD14} ,we conclude there exists a Hilbert newform~$g \in S_2(\fN')^{\rm new}$ 
such that $\rhobar_{\ff_{11},\fp_0} \simeq \rhobar_{g, \fp'}$ for some prime $\fp' \mid 7$ in the coefficient field~$\Q_g$ of~$g$.

For $\fN'=(1)$ and $\fN' = \fq_{13}$, we have $\dim S_2(\fN')^{\rm new} = 0$, so $\fN' = (2)$ or $2\fq_{13}$.

For $\fN' = (2)$, we have $\dim S_2(\fN')^{\rm new} = 1$. 
The unique newform corresponds to the isogeny class of the 
base change to $K$ of the elliptic curve $W$ with Cremona 
label~$338b1$. From LMFDB~\cite[Elliptic Curves over $\Q$]{lmfdb}, we see that 
$\rhobar_{W,7}$ is reducible, hence $\rhobar_{\ff_{11},\fp_0}\not \simeq \rhobar_{W, 7}|_{G_K}$.

We conclude that $\rhobar_{\ff_{11},\fp_0}$ ramifies at both $(2)$ and $\fq_{13}$, that is $\fN'=2\fq_{13}$.

Now, a similar argument to the one used in the proof of Proposition~\ref{P:reducible} shows that $\rhobar_{\ff_{11},\fp_0}$ extends to an irreducible representation 
$\rhobar : G_\Q \to \GL_2(\F_7)$, and that there is a classical newform $h \in S_2(26)^{\rm new}$ such that $\rhobar \simeq \rhobar_{h,\fp}$ for some prime
$\fp \mid 7$ in $\Q_h$. There are two newforms $h_1, h_2 \in S_2(26)^{\rm new}$, with
rational coefficients, corresponding to the isogeny classes of the  elliptic curves $W_1$ and $W_2$ with Cremona label~$26a1$ and $26b1$ respectively. 

The curve $W_2$ has a 7-torsion point, so $\rhobar_{h_2,7}$ is reducible, hence
$\rhobar \not \simeq \rhobar_{h_2,7}$. 

Finally, we observe that $a_5(W_1) = -3$ for the prime $5$, which is split in $K$. So, for a prime $\fq \mid 5$ in $K$, we have 
\[
 \Tr(\rhobar_{h_1,7}|_{G_K}(\Frob_\fq)) = \Tr(\rhobar_{W_1,7}(\Frob_5)) \equiv -3 \pmod{7}.
\]
But, we easily check that $\Tr(\rhobar_{\ff_{11},\fp_0}(\Frob_\fq)) \equiv 6 \not \equiv - 3 \pmod{\fp_0}$, showing that $\rhobar \not \simeq \rhobar_{h_1,7}$. 

Since there are no other possible forms~$h$ we obtain the desired contradiction.
\end{proof}


\begin{thebibliography}{999}

\bibitem{AnniSiksek} S.\ Anni and S.\ Siksek,
{\em Modular elliptic curves over real abelian fields and the generalized {F}ermat equation $x^{2\ell} + y^{2m}=z^p$}, Algebra \& Number Theory {\bf 10} (2016), no.6, 1147--1172

\bibitem{BennetSkinner} M.\ A.\ Bennett and C.\ M.\ Skinner,
{\em Ternary Diophantine equations via Galois representations and modular forms}, Canad.\ J.\ Math.\ {\bf 56} (2004), no. 1, 23--54.
 

\bibitem{programs} N. Billerey, I. Chen, L. Demb\'el\'e, L. Dieulefait, and N. Freitas, 
{\em Supporting {{\tt Magma}} program files for this paper},
available as ancillary files of the preprint version of this work on arxiv
\url{https://arxiv.org/abs/1802.04330}

\bibitem{BCDF} N. Billerey, I. Chen. L. Dieulefait, and N. Freitas, 
{\em A multi-Frey approach to Fermat equations of signature $(r,r,p)$},
Trans. Amer. Math. Soc. {\bf 371} (2019), no. 12, 8651--8677.

\bibitem{magma}
W. Bosma, J. Cannon, and C. Playoust.
\newblock The {M}agma algebra system. {I}. {T}he user language.
J. Symbolic Comput. {\bf 24} (1997), no. 3--4, 235--265.
\newblock Computational algebra and number theory (London, 1993).

\bibitem{BD14} C.~Breuil, and F.~Diamond,
{\em Formes modulaires de Hilbert modulo $p$ et valeurs d'extensions entre caract\`eres galoisiens},
Ann. Sci. \'Ec. Norm. Sup\'er. (4) {\bf 47} (2014), no. 5, 905--974.

\bibitem{BMSI} Y.\ Bugeaud, M.\ Mignotte and S.\ Siksek
{\em A multi-Frey approach to some multi-parameter families of Diophantine equations},
Canadian Journal of Mathematics {\bf 60} (2008), 491-519.


\bibitem{BH06} C.\ J.\ Bushnell, G.\ Henniart, 
{\em The local Langlands conjecture for $\GL(2)$}, 
{\it Grundlehren der Mathematischen Wissenschaften} {\bf 335},
Springer-Verlag, Berlin, 2006. xii+347 pp.


\bibitem{Cali} \'E.\ Cali, 
{\em D{\'e}faut de semi-stabilit{\'e} des courbes elliptiques dans le cas non ramifi{\'e}}, 
Canad.\ J.\ Math.\ {\bf 56} (2004), no 4, 673--698.

\bibitem{car86}  H. Carayol, 
{\em Sur les repr\'esentations $l$-adiques associ\'ees aux formes modulaires de Hilbert},
Ann. Sci. \'Ecole Norm. Sup. (4) {\bf 19} (1986), no. 3, 409--468.



\bibitem{CK} I.\ Chen and A.\ Koutsianas, {\em A modular approach to {F}ermat equations of signature $(p,p,5)$ using {F}rey hyperelliptic curves}, preprint.

\bibitem{cmsv} E. Costa, N. Mascot, J. Sijsling, J. Voight,
{\em Rigorous computation of the endomorphism ring of a {J}acobian},
Math. Comp. {\bf 88} (2019), no. 317, 1303--1339. 



\bibitem{DahmenSiksek} S. R. Dahmen, and S. Siksek,
{\em Perfect powers expressible as sums of two fifth or seventh powers}, 
Acta Arith. {\bf 164} (2014), no. 1, 65--100. 

\bibitem{DarmonDuke} H.\ Darmon,
{\em Rigid local systems, {H}ilbert modular forms, and {F}ermat's last theorem}, 
Duke Math.\ J.\ {\bf 102} (2000), no 3, 413--449.

\bibitem{dem20}L.~Demb\'el\'e,
{\em An intriguing hyperelliptic Shimura curve quotient of genus $16$},
Algebra Number Theory {\bf 14} (2020), no. 10, 2713--2742.

\bibitem{DV13} L.~Demb\'el\'e, and J.~Voight, 
{\em Explicit methods for Hilbert modular forms.} Elliptic curves, Hilbert modular forms and Galois deformations, 
135--198, Adv. Courses Math. CRM Barcelona, Birkh\"auser/Springer, Basel, 2013.



\bibitem{dk16}
L. Demb\'el\'e, and A. Kumar, 
{\em Examples of abelian surfaces with everywhere good reduction},
Math. Ann. {\bf 364} (2016), no. 3-4, 1365--1392.


 
\bibitem{dia95} F. Diamond, 
{\em An extension of Wiles' results}, 
Modular forms and Fermat's last theorem (Boston, MA, 1995), 475--489, Springer, New York, 1997. 
 


\bibitem{BP} J.\  Burgos Gil and A.\ Pacetti, 
{\em Hecke and Sturm bounds for Hilbert modular forms over real quadratic fields},
Math. Comp. {\bf 86} (2017), no. 306, 1949--1978. 

\bibitem{Deligne-Rapoport} P. Deligne and M. Rapoport, {\em Les sch\'emas de modules de courbes elliptiques}, Proc. Antwerpen Conference, vol. 2, Lecture Notes in Mathematics, Volume 349, 1973, pp. 143-316, Springer-Verlag.

\bibitem{DD17} T. Dokchitser, C. Doris
{\em $3$-torsion and conductor of genus 2 curves},
Math. Comp. {\bf 88} (2019), no. 318, 1913--1927. 

\bibitem{Dor17} C. Doris, Package for computing the conductor or genus $2$ curves.
Available at \url{https://cjdoris.github.io/Genus2Conductor/\#installation}

\bibitem{ek14}
N. Elkies, and A. Kumar, 
{\em $K3$ surfaces and equations for Hilbert modular surfaces},
Algebra Number Theory {\bf 8} (2014), no. 10, 2297--2411.
 

\bibitem{ell05} J. Ellenberg, {\em Serre's conjecture over $\F_9$},
Ann. of Math. (2) {\bf 161} (2005), no. 3, 1111--1142.


\bibitem{feit} W.\ Feit, 
{\em The Representation Theory of Finite Groups},
North-Holland, Amsterdam-New York-Oxford, 1982.



\bibitem{F33p} N. Freitas,
{\em On the Fermat-type equation $x^3 + y^3 = z^p$},
Commentarii Mathematici Helvetici {\bf 91} (2016), 295--304.



\bibitem{FLHS} N. Freitas, B. Le Hung, S. Siksek,
{\em Elliptic Curves over Real Quadratic Fields are Modular},
Inventiones Mathematicae {\bf 201} (2015), no. 1, 159--206.


\bibitem{GV11}
M.~Greenberg, and J.~Voight, 
{\em Computing systems of Hecke eigenvalues associated to Hilbert modular forms},
Math. Comp. {\bf 80} (2011), no. 274, 1071--1092.


\bibitem{Hida} Haruzo Hida, {\em $p$-adic Automorphic Forms on Shimura Varieties}, Springer Monographs in Mathematics, Springer-Verlag, 2004.

\bibitem{kw09} C.~Khare, and J.-P. Wintenberger, 
{\em On Serre's conjecture for $2$-dimensional mod p representations of ${\rm Gal}(\overline{\Q}/\Q)$}, 
Ann. of Math. (2) {\bf 169} (2009), no. 1, 229--253.

\bibitem{kis09} M. Kisin, {\em Moduli of finite flat group schemes, and modularity},
Ann. of Math. (2) {\bf 170} (2009), no. 3, 1085--1180.

\bibitem{Kraus1990} A.\ Kraus,
{\em Sur le d\'efaut de semi-stabilit\'e des courbes elliptiques \`a r\'eduction additive},
Manuscripta Math.\ {\bf 69} (1990), no. 4, 353--385.

\bibitem{Kraus1998} A.\ Kraus,
{\em Sur l'\'equation $a^3 + b^3 = c^p$},
Experimental Math.\ {\bf 7} (1998), no. 1, 1--13.

\bibitem{kut80} P. Kutzko, 
{\em The Langlands conjecture for $\GL_2$ of a local field}, 
Ann. of Math. (2) {\bf 112} (1980), no. 2, 381--412.

\bibitem{liu93} Q. Liu, 
{\em Courbes stables de genre 2 et leur sch\'ema de modules},
Math. Ann. {\bf 295} (1993), no. 2, 201--222. 

\bibitem{lmfdb} The {LMFDB Collaboration},
{\em The L-functions and Modular Forms Database},
\url{http://www.lmfdb.org}, 2013.

\bibitem{martin} K.~Martin, 
{\em The Jacquet-Langlands correspondence, Eisenstein congruences, and integral $L$-values in weight $2$}, 
Math. Res. Lett. {\bf 24} (2017), no. 6, 1775--1795.

\bibitem{mestre}J.-F. Mestre. {\em Construction de courbes de genre 2 a partir de leurs modules.} In Effective Meth-
ods in Algebraic Geometry (Castiglioncello, 1990), volume 94 of Progr. Math. Birkh\"auser, 1991, pp. 313--334.

\bibitem{nek12} J.~Nekov\'a\v{r}, 
\emph{{L}evel {R}aising and {A}nticyclotomic {S}elmer {G}roups for {H}ilbert {M}odular {F}orms of
{W}eight {T}wo}, 
Canad. J. Math. Vol. {\bf 64} (3), 2012, 588--668.


\bibitem{ribet10}
K. Ribet, {\em Non-optimal levels of mod $\ell$ reducible Galois representations}, 
CRM Lecture notes, 2010. Available at \url{http://math.berkeley.edu/~ribet/crm.pdf}.

\bibitem{ribet} K. Ribet, {\em Abelian varieties over $\Q$ and modular forms},
Modular curves and abelian varieties, 241--261, Progr. Math. {\bf 224}, Birkh\"auser, Basel, 2004.

\bibitem{SerreReps} J.-P.\ Serre, 
{\em Linear representations of finite groups},
Springer-Verlag, New York-Heidelberg, 1977.
Translated from the second French edition by Leonard L. Scott,
Graduate Texts in Mathematics, Vol. 42. 
 
\bibitem{ST1968} J.-P.\ Serre and J.\ Tate,
{\em Good reduction of abelian varieties},
Annals of Math. {\bf 88} (1968), no. 2, 492--517.

\bibitem{shi78} Goro Shimura, 
\emph{The special values of the zeta functions associated with 
{H}ilbert modular forms}, Duke Math. J. \textbf{45} (1978), no.~3, 637--679.



\bibitem{Shnidman} Ari Shnidman, \emph{Quadratic twists of abelian varieties with real multiplication}, International Math. Research Notices, \textbf{2021} (2021), Issue 5, 3267--3298.


\bibitem{wilson98} J.\ Wilson,
{\em Curves of genus 2 with real multiplication by a square root of 5}, PhD thesis, University of Oxford, 1998.

\bibitem{yoo}H.~Yoo, {\em Non-optimal levels of a reducible {$\rm{mod}\,\ell$} modular representation},
Trans. Amer. Math. Soc. {\bf 371} (2019), no. 6,  3805--3830.



\bibitem{tay89}
R. Taylor, 
\emph{On {G}alois representations associated to {H}ilbert modular forms}, 
Invent.\ Math.\ \textbf{98} (1989), no.~2, 265--280.
  
 \bibitem{zha01}S.~Zhang, 
 \emph{ {H}eights of {H}eegner {P}oints on {S}himura {C}urves},
Annals of Math., Second Series, {\bf 153}, no. 1 (2001), 27--147.

\end{thebibliography}
\end{document}